\documentclass[11pt]{article}

\usepackage[top=1in, left=1.20in, bottom=1in, right=1.20in]{geometry}

\usepackage{hyperref}  %provides clickable links to references within the document
\usepackage{setspace} 
\usepackage{color}
\usepackage{graphicx}
\usepackage{amsfonts}
\usepackage{amssymb}
\usepackage{amsmath}
\usepackage{amsthm}

\newtheorem{theorem}{Theorem}[section]

\newtheorem{lemma}[theorem]{Lemma}
\newtheorem{conjecture}[theorem]{Conjecture}
\newtheorem{claim}{}[theorem]

\newcommand{\del}{\backslash}

\title{Graphs with girth $2\ell+1$ and without longer odd holes that contain an odd $K_4$-subdivision }

\author{Rong Chen, Yidong Zhou\\
\\
Center for Discrete Mathematics,\ \ Fuzhou University\\
Fuzhou,\ \ P. R. China}
\begin{document}

\maketitle

\footnote{

Emails: rongchen13@163.com (R. Chen),\ \ zoed98@126.com (Y. Zhou).

This research was partially supported by grants from the National Natural Sciences Foundation of China (No. 11971111).
}

\begin{abstract}
%We say a graph $H$ is an {\em odd $K_4$-subdivision} if $H$ is isomorphic to a $K_4$-subdivision and all face cycles of $H$ are odd holes.
We say that a graph $G$ has an {\em odd $K_4$-subdivision} if some subgraph of $G$ is isomorphic to a $K_4$-subdivision which if embedded in the plane the boundary of each of its faces has odd length and is an induced cycle of $G$. %and whose faces are all odd holes of $G$.
%{\blue In a $K_4$-subdivision $H$, We say a cycle $C$ is a {\em face} of $H$ if $C$ is the frontier of a face of $H'$, where $H'$ is a planar embedding of $H$. We say that a graph $G$ has an {\em odd $K_4$-subdivision} if some subgraph of $G$ is isomorphic to a $K_4$-subdivision and whose faces are all odd holes of $G$.}
For a number $\ell\geq 2$, let $\mathcal{G}_{\ell}$ denote the family of graphs which have girth $2\ell+1$ and have no odd hole with length greater than $2\ell+1$. Wu, Xu and Xu conjectured that every graph in $\bigcup_{\ell\geq2}\mathcal{G}_{\ell}$ is 3-colorable.
Recently, Chudnovsky et al. and Wu et al., respectively, proved that every graph in $\mathcal{G}_2$ and $\mathcal{G}_3$ is 3-colorable. In this paper, we prove that no  $4$-vertex-critical graph  in $\bigcup_{\ell\geq5}\mathcal{G}_{\ell}$ has an odd $K_4$-subdivision.
Using this result, Chen proved that all graphs in $\bigcup_{\ell\geq5}\mathcal{G}_{\ell}$ are 3-colorable.
%Chudnovsky and Seymour proved that every graph in $\mathcal{G}_2$ is 3-colorable. Wu, Xu, and Xu proved that every graph in $\mathcal{G}_3$ is 3-colorable. To prove that a graph $G\in\mathcal{G}_2\cup\mathcal{G}_3$ is 3-colorable,

\end{abstract}

\section{Introduction}

All graphs considered in this paper are finite, simple, and undirected.
A {\em proper coloring} of a graph $G$ is an assignment of colors to the vertices of $G$ such that no two adjacent vertices receive the same color. A graph is {\em $k$-colorable} if it has a proper coloring using at most $k$ colors. The {\em chromatic number} of $G$, denoted by $\chi(G)$, is the minimum number $k$ such that $G$ is $k$-colorable.
%{\blue A vertex is called to be a {\em degree-k vertex}, if it has exactly k neighbours. A graph is said to be a {\em plane graph}, if it can be drawn in the plane so that its edges intersect only at their ends. A plane graph $G$ partitions the rest of the plane into a number of arcwise-connected open sets. These sets are called the {\em faces} of $G$. The {\em boundary} of a face $f$ is the boundary of the open set $f$ in the usual topological sense. And a cycle of a plane graph $G$ is called a {\em face cycle} if it is a boundary of a face of $G$.}
%Given a graph with a large chromatic number, it is natural to ask whether it must contain induced subgraphs with particular properties. A family $\mathcal{G}$ of graphs is said to be $\chi$-$bounded$ if there exists a function $f$ such that for every graph $\in \mathcal{G}$, every induced subgraph $H$ of $G$ satisfies $\chi(H)\leq f(\omega(H))$. The function $f$ is called a $\chi$-$binding$ function for $\mathcal{G}$. The notion of $\chi$-bounded families was introduced by Gy$\acute{\mathrm{a}}$rf$\acute{\mathrm{a}}$s in \cite{AG}. We refer the reader to the recent survey by Scott and Seymour \cite{AS20} for various nice results.

The {\em girth} of a graph $G$, denoted by $g(G)$, is the minimum length of  cycles in $G$. A {\em hole} in a graph is an induced cycle of length at least four. An {\em odd hole} means a hole of odd length. For any integer $\ell\geq2$, let $\mathcal{G}_{\ell}$ be the family of graphs that have girth $2\ell + 1$ and have no odd holes of length at least $2\ell + 3$.  Robertson conjectured in \cite{ND11} that the Petersen graph is  the only graph in $\mathcal{G}_2$ that is 3-connected and internally 4-connected.
Plummer and Zha \cite{PM14} disproved Robertson's conjecture and proposed the conjecture that all 3-connected and internally 4-connected graphs in $\mathcal{G}_2$ have bounded chromatic numbers, and proposed the strong conjecture that such graphs are 3-colorable. The first was proved by Xu, Yu, and Zha \cite{XB17}, who proved that all graphs in $\mathcal{G}_2$ are 4-colorable. The strong conjecture proposed by Plummer and Zha in \cite{PM14} was solved by Chudnovsky and Seymour \cite{MC22}.
%Chudnovsky and Seymour \cite{MC22} proved that every graph in $\mathcal{G}_2$ is 3-colorable.In the same paper, Chudnovsky and Seymour also gave a short and pretty proof of the result in \cite{XB17}.
Wu, Xu, and Xu \cite{WD2204} showed that graphs in $\bigcup_{\ell\geq2}\mathcal{G}_{\ell}$  are 4-colorable and conjectured
\begin{conjecture}\label{conj1}(\cite{WD2204}, Conjecture 6.1.)
For each integer $\ell\geq2$, every graph in $\mathcal{G}_{\ell}$ is $3$-colorable.
\end{conjecture}
\noindent Wu, Xu and Xu \cite{WD22} recently proved that Conjecture \ref{conj1} holds for $\ell=3$.
%{\blue By the similar method as in this paper, Wang and Wu \cite{WY23} showed that Conjecture \ref{conj1} holds for $\ell=4$.}
%We say that a graph $H$ is an {\em odd $K_4$-subdivision} if $H$ is isomorphic to a $K_4$-subdivision and whose face cycles are all odd holes.

%{\red Let $H=\{u_1,u_2,u_3,u_4,P_1,P_2,Q_1,Q_2,L_1,L_2\}$ be a $K_4$-subdivision such that $u_1,u_2,u_3,u_4$ are degree-3 vertices of $H$ and $P_1$ is a $(u_1,u_2)$-path, $P_2$ is a $(u_3,u_4)$-path, $Q_1$ is a $(u_2,u_3)$-path, $Q_2$ is a $(u_1,u_4)$-path, $L_1$ is a $(u_1,u_3)$-path and $L_2$ is a $(u_2,u_4)$-path. Let $C_1:=P_1\cup Q_1\cup L_1$, $C_2:=P_1\cup Q_2\cup L_2$, $C_3:=P_2\cup Q_1\cup L_2$ and $C_4:=P_2\cup Q_2\cup L_1$ be four holes in $H$. We call that $H$ is an {\em odd $K_4$-subdivision} if $C_1,C_2,C_3$ and $C_4$ are odd holes. We say that a graph $G$ has an {\em odd $K_4$-subdivision} if some subgraph of $G$ is isomorphic to an odd $K_4$-subdivision.}

%{\blue In a $K_4$-subdivision $H$, We say a cycle $C$ is a {\em face} of $H$ if $C$ is the frontier of a face of $H'$, where $H'$ is a planar embedding of $H$. We say that a graph $G$ has an {\em odd $K_4$-subdivision} if some subgraph of $G$ is isomorphic to a $K_4$-subdivision and whose faces are all odd holes of $G$.}
We say that a graph $G$ has an {\em odd $K_4$-subdivision} if some subgraph of $G$ is isomorphic to a $K_4$-subdivision which if embedded in the plane the boundary of each of its faces has odd length and is an induced cycle of $G$.
Note that an odd $K_4$-subdivision of $G$ maybe not induced. However, when $G\in\mathcal{G}_{\ell}$ for each integer $\ell\geq2$, all odd $K_4$-subdivisions of $G$ are induced by Lemma \ref{odd k4} (2).  %To prove that a graph $G\in\mathcal{G}_2$ is 3-colorable, Chudnovsky and Seymour \cite{MC22} first proved that no  $4$-vertex-critical graph $G\in \mathcal{G}_2$ has an odd $K_4$-subdivision (Odd $K_4$-subdivisions of a graph in $\mathcal{G}_2$ are called $\mathcal{P}_2$  in \cite{MC22}), and then proved that $G$ has such a subgraph. Following the ideas in \cite{MC22}, Wu, Xu, and Xu \cite{WD22} proved that Conjecture \ref{conj1} holds for $\ell=3$.
In this paper, we prove the following theorem.

\begin{theorem}\label{main thm}
No  $4$-vertex-critical graph in $\bigcup_{\ell\geq5}\mathcal{G}_{\ell}$ has an odd $K_4$-subdivision. %  for each number $\ell\geq5$.
\end{theorem}
%Let $\ell\geq5$ be a number, and $G$ be a graph in $\mathcal{G}_{\ell}$. If $G$ is $4$-vertex-critical,  then $G$ has no odd $K_4$-subdivision.

Using Theorem \ref{main thm}, Chen \cite{Chen22} proved that Conjecture \ref{conj1} holds for  all $\ell\geq5$. Recently, following idea in this paper and \cite{Chen22}, Wang and Wu \cite{WW23} further proved that Conjecture \ref{conj1} holds for $\ell=4$.

\section{Preliminaries}
A {\em cycle} is a connected $2$-regular graph. Let $G$ be a graph. A vertex $v\in V(G)$ is called a {\em degree-$k$ vertex} if it has exactly k neighbours.
For any $U\subseteq V(G)$, let $G[U]$ be the subgraph of $G$ induced on $U$. For subgraphs $H$ and $H'$ of $G$, set $|H|:=|E(H)|$ and $H\Delta H':=E(H)\Delta E(H')$.
Let $H\cup H'$ denote the subgraph of $G$ whose vertex set is $V(H)\cup V(H')$ and edge set is $E(H)\cup E(H')$. Let $H\cap H'$ denote the subgraph of $G$ with edge set  $E(H)\cap E(H')$ and without isolated vertex. Let $N(H)$ be the set of vertices in $V(G)-V(H)$ that have a neighbour in $H$. Set $N[H]:=N(H)\cup V(H)$.
%A {\em theta graph} is a graph that consists of a pair of distinct vertices joined by three internally disjoint paths. If two cycles of a theta subgraph are odd holes, we say that the theta subgraph is {\em odd}. Evidently, each theta subgraph is induced. Let $H$ be a $K_4$-subdivision. When all cycles in $H$ are odd, we say that $H$ is {\em odd}. %When $H$ has exactly two odd cycles, we say that $H$ is {\em balanced}.

Let $P$ be an $(x,y)$-path and $Q$ be a $(y,z)$-path.
When $P$ and $Q$ are internally disjoint, % with a common end vertex,
let $PQ$ denote the $(x,z)$-path $P\cup Q$. Evidently, $PQ$ is a path when $x\neq z$, and $PQ$ is a cycle when $x=z$.
Let $P^*$ denote the set of internal vertices of $P$. When $u,v\in V(P)$, let $P(u,v)$ denote the subpath of $P$ with ends $u, v$. For simplicity, we will let $P^*(u,v)$ denote $(P(u,v))^*$.

%It is easy to prove that Lemma \ref{2-edge-cut-1} is true.
%
%\begin{lemma}\label{2-edge-cut-1}
%Let $X$ be a $2$-edge cut of a graph $G$. If $\chi(G\del X)\geq3$, then $\chi(G)=\chi(G\del X)$.
%\end{lemma}
%
%For a number $k\geq2$, a graph $G$ is {\em $k$-vertex-critical} if $\chi(G)=k$ and $\chi(G-v)\leq k-1$ for each vertex $v$ of $G$.
%
%\begin{lemma}\label{2-edge-cut}
%For a number $k\geq4$, each $k$-vertex-critical graph $G$ has no $2$-edge-cut.
%\end{lemma}
%\begin{proof}
%Assume to the contrary that $G$ has a $2$-edge-cut $X$. Let $G_1, G_2$ be the components of $G\del X$. Then $G_1, G_2$ are $(k-1)$-colorable. If $\chi(G_1\cup G_2)\geq3$, then $G$ is $(k-1)$-colorable by Lemma \ref{2-edge-cut-1}, which is not possible. If $\chi(G_1\cup G_2)\leq2$, then $G$ is 3-colorable, which is a contradiction.
%\end{proof}
A graph is $k$-{\em vertex-critical} if $\chi(G)=k$ but $\chi(G\setminus v)<k$ for each $v\in V(G)$.
Dirac in \cite{DG53} proved that every $k$-vertex-critical graph is $(k-1)$-edge-connected. Hence, we have

\begin{lemma}\label{2-edge-cut}
For each integer $k\geq4$, each $k$-vertex-critical graph $G$ has no $2$-edge-cut.
\end{lemma}

%\begin{theorem}\label{penta}
%Every graph $G$ in $\mathcal{G}_2$ is three colorable.
%\end{theorem}
A {\em theta graph} is a graph that consists of a pair of distinct vertices joined by three internally disjoint paths. Let $C$ be a hole of a graph $G$. A path $P$ of $G$ is a {\em chordal path} of $C$ if $V(P^*)\cap V(C)=\emptyset$ and $C\cup P$ is an induced theta-subgraph of $G$.  Lemma \ref{easy case} will be frequently used. %Since $g(G)=2\ell+1$ and each odd hole of $G$ has length $2\ell+1$, by simple computation we can prove that Lemma \ref{easy case} is true and this lemma will be frequently used.

\begin{lemma}\label{easy case}
Let $\ell\geq 2$ be an integer and $C$ be an odd hole of a graph $G\in \mathcal{G}_{\ell}$. Let $P$ be a chordal path of $C$, and $P_1,P_2$ be the internally disjoint paths of $C$ that have the same ends as $P$. Assume that $|P|$ and $|P_1|$ have the same parity.  If $|P_1|\neq 1$, then $|P_1|>|P_2|$
%{\blue $|P_2|\leq \ell \leq \ell+1\leq|P_1|=|P|$}
and all chordal paths of $C$ with the same ends as $P_1$ have length $|P_1|$.
\end{lemma}

\begin{proof}
Since $|C|=2\ell+1$, $|P_1|\neq 1$ and $|P|$ and $|P_1|$ have the same parity, $P\cup P_2$ is an odd hole. Moreover, since $g(G)=2\ell+1$ and all odd holes in $G$ have length $2\ell+1$, we have $\ell+1\leq|P_1|=|P|$ and $|P_2|\leq \ell$, %{\blue $|P_2|\leq \ell \leq \ell+1\leq|P_1|=|P|$},
so $|P_1|>|P_2|$ and all chordal paths of $C$ with the same ends as $P_1$ have length $|P_1|$.
%$|P_2|<|P_1|$ with $\ell+1\leq|P_1|$ and with $|P_2|\leq \ell$.
\end{proof}

Let $P$ be a path with $i$ vertices. If $G-V(P)$ is disconnected, then we say that $P$ is a {\em $P_i$-cut}. Usually, a $P_2$-cut is also called a {\em $K_2$-cut}. Evidently, every $k$-vertex-critical graph has no $K_2$-cut. Chudnovsky and Seymour in \cite{MC22} proved that every $4$-vertex-critical graph $G$ in $\mathcal{G}_2$ has no $P_3$-cut. Using the same argument as \cite{MC22}, Wu et al. \cite{WD22} extend this result to graphs in $\bigcup_{\ell\geq2}\mathcal{G}_{\ell}$. Since the paper \cite{WD22} does not include a proof of Lemma \ref{P3}, we give a proof here for completeness. 
%methods can also be used to prove the following lemma.
\begin{lemma}\label{P3}(\cite{WD22})
For any number $\ell\geq2$, every $4$-vertex-critical graph in $\mathcal{G}_{\ell}$ has neither a $K_2$-cut nor a $P_3$-cut.
\end{lemma}
\begin{proof}
It is well-known that every $k$-vertex-critical graph has no clique as a cut. Hence, it suffice to show that every $4$-vertex-critical graph in $\mathcal{G}_{\ell}$ has no $P_3$-cut. Let $G\in\mathcal{G}_{\ell}$ be a $4$-vertex-critical graph. Assume to the contrary that $P=v_1v_2v_3$ is a path such that $G\backslash \{v_1,v_2,v_3\}$ is disconnected. Since $G$ has no $K_3$ as its cut, $v_1v_3\notin E(G)$. Let $A_1$ be the a component of $G\backslash \{v_1,v_2,v_3\}$, and let $A_2$ be the union of all other components. Set $G_i:=G[A_i\cup \{v_1,v_2,v_3\}]$ for $i=1,2$.
Since $G$ is $4$-vertex-critical, both $G_1$ and $G_2$ are 3-colorable. Let $\phi_i:V(G_i)\rightarrow\{1,2,3\}$ be a 3-coloring for $i=1,2$. By symmetry we may assume that $\phi_i(v_1)=1$ and $\phi_i(v_2)=2$ for $i=1,2$. Thus $\phi_1(v_3),\phi_2(v_3)\in \{1,3\}$. If $\phi_1(v_3)=\phi_2(v_3)$, then $G$ is 3-colorable, which is a contradiction. Thus by symmetry we may assume that $\phi_1(v_3)=1$ and $\phi_2(v_3)=3$. Let $H_1$ be the subgraph of $G_1$ induced on the set of vertices $v\in V(G_1)$ with $\phi_1(v)\in \{1,3\}$. If $v_1,v_3$ belong to different components of $H_1$, then by exchanging colors in the component containing $v_3$, we obtain another 3-coloring of $G_1$ that can be combined with $\phi_2$ to show that $G$ is 3-colorable. So $v_1,v_3$ belong to the same component of $H_1$. Then there is an induced $(v_1, v_3)$-path $P_1$ in $H_1$ having even length as $\phi_1(v_1)=1=\phi_1(v_3)$. Similarly, there is an induced $(v_1, v_3)$-path $P_2$ in $G_2$ having odd length as $\phi_2(v_1)=1$ and $\phi_2(v_3)=3$. Moreover, since $PP_1,PP_2$ are cycles of $G$ and $g(G)=2\ell+1$, we have $|P_1|\geq 2\ell-1$ and $|P_2|\geq 2\ell$, so $P_1\cup P_2$ is an odd hole of $G$ of length at least $4\ell-1$, which is a contradiction as $G\in\mathcal{G}_{\ell}$.
\end{proof}

\begin{lemma}\label{same length}
Let $\ell\geq 2$ be an integer and $x,y$ be non-adjacent vertices of a graph $G\in \mathcal{G}_{\ell}$. Let $P$ be an induced $(x,y)$-path of $G$. If $|P|\leq \ell$ and all induced $(x,y)$-paths have length $|P|$, then no block of $G$ contains two non-adjacent vertices in $V(P)$. In particular, each vertex in $P^*$ is a cut-vertex of $G$.
\end{lemma}
\begin{proof}
Assume not. Then there is a block $B$ of $G$ containing two consecutive edges of $P$. Let $Q$ be an induced path in $B$ with ends in $V(P)$ and with $V(P)\cap V(Q^*)=\emptyset$. Since every pair of edges in a 2-connected graph is contained in a cycle, such a $Q$ exists. %Let $C$ be the unique cycle in $P\cup Q$. %Such $Q$ obviously exists.Then $P\cup Q$ has a unique cycle $C$.
Without loss of generality we may further assume that $Q$ is chosen with $|Q|$ as small as possible. Let $C$ be the unique cycle in $P\cup Q$. Then $C\Delta P$ is an $(x,y)$-path. Since $Q$ is induced, the ends of $Q$ are not adjacent. Moreover, since $Q$ is chosen with $|Q|$ as small as possible, $C\Delta P$ is an induced $(x,y)$-path, so $|C\Delta P|=|P|\leq \ell$ by the assumption of the lemma. Hence, $|C|\leq 2\ell$, contrary to the fact $g(G)=2\ell+1$.
%{\blue Let $u,v$ be the ends of $Q$, and $P'$ be the subpath of $P$ with ends $u,v$. We claim that there is no edge between $V(P'\Delta P)$ and $V(Q^*)$. Suppose to the contrary that we may assume that there exists a such edge, say $e=ab$, such that $a\in V(P\Delta P')$ and $b\in V(Q^*)$. By the minimality of $Q$, we have $b$ is one of the ends of $Q^*$, say $ub\in E(G)$ by symmetry. But then, $abu\cup P(a,u)$ is a cycle with length less than $\ell +2$, a contradiction. That is,}
%Then  $C\Delta P$ is an induced $(x,y)$-path as the fact that $Q$ is induced implies that the ends of $Q$ are not adjacent. % where $C$ is the unique cycle in $P\cup Q$. Since $|C\Delta P|=|P|\leq \ell$ by the assumption, we have $|C|\leq 2\ell$, contrary to the fact $g(G)=2\ell+1$.
\end{proof}

\begin{lemma}\label{2-vertex-cut}
Let $\ell\geq 4$ be an integer and $x,y$ be non-adjacent vertices of a graph $G\in \mathcal{G}_{\ell}$.  Let $X$ be a vertex cut of $G$ with $\{x,y\}\subseteq X\subseteq N[\{x,y\}]$, and $G_1$ be an induced subgraph of $G$ whose vertex set consists of $X$ and the vertex set of a component of $G-X$. If all induced $(x,y)$-paths in $G_1$ have  length $k$ with $4\leq k\leq\ell$, then $G$ has a degree-$2$ vertex, a $K_1$-cut, or a $K_2$-cut.
\end{lemma}
\begin{proof}
Assume that $G$ has no degree-2 vertices. Let $P$ be an induced $(x,y)$-path in $G_1$. Let $uvw$ be a subpath of $P^*$. Such $uvw$ exists as $k\geq4$. By the definition of $G_1$, we have $v\notin X$, so $N_G[v]=N_{G_1}[v]$, which implies $d_{G_1}(v)\geq3$. %Moreover, since $u,v,w$ are cut-vertices of $G_1$
By applying Lemma \ref{same length} to $G_1$, there is a block $B$ of $G_1$ such that either $V(B)\cap V(P)=\{v\}$, or $B$ is not isomorphic to $K_2$ and $V(B)\cap V(P)$ is $\{u, v\}$ or $\{u, v\}$. When the first case happens, since $X\subseteq N[\{x,y\}]$, $x,y\notin B$ and $v\notin X$, we have $X\cap V(B)=\emptyset$, for otherwise $P(x,v)$ or $P(v,y)$ is contained in a cycle of $P\cup B$,
%$P(x,v)$ (or $P(v,y)$) is contained in a cycle of  $P(x,v)\cup B$ (or $P(v,y)\cup B$),
so the vertex $v$ is a cut-vertex of $G$ as $X$ is a vertex cut of $G$. When the latter case happens, by symmetry we may assume that $V(B)\cap V(P)=\{u, v\}$. Since $B$ is a block of $G_1$, $X\subseteq N[\{x,y\}]$ and $uvw$ is a subpath of $P^*$, we have $V(B)\cap X=\{u\}\cap X$, so $\{u, v\}$ is a $K_2$-cut of $G_1$ and $G$.
\end{proof}

\begin{figure}[htbp]
\begin{center}
\includegraphics[height=6cm]{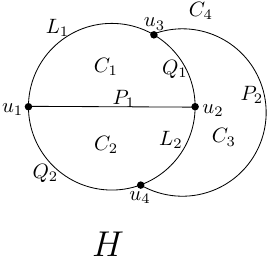}
\caption{$u_1,u_2,u_3,u_4$ are the degree-3 vertices of $H$. All faces $C_1, C_2, C_3,C_4$ of $H$ are odd holes. $\{P_1,P_2\}$, $\{Q_1,Q_2\}$, $\{L_1,L_2\}$ are the pairs of vertex disjoint arrises of $H$.}
\label{H}
\end{center}
\end{figure}

Let $H$ be a graph that is isomorphic to a  subdivision of $K_4$, and let $P$ be a path of $H$ whose ends are degree-3 vertices in $H$. If $P^*$ contains no degree-3 vertex of $H$, then we say that $P$ is an {\sl arris} of $H$.  Evidently, there are exactly six arrises of $H$. See Figure \ref{H}.

\begin{lemma}\label{odd k4}
For any integer $\ell\geq 2$, if  a graph $G\in \mathcal{G}_{\ell}$ has an an odd $K_4$-subdivision $H$, then the following statements hold.
%Let $H$ be a subgraph of a graph $G\in \mathcal{G}_{\ell}$. If $H$ is isomorphic to an odd $K_4$-subdivision, then the following statements hold.
\begin{itemize}
\item[(1)] Each pair of vertex disjoint arrises in $H$ have the same length and their lengths are at most $\ell$.
\item[(2)] $H$ is an induced subgraph of $G$.
\item[(3)] When  $\ell\geq3$, every vertex in $V(G)-V(H)$ has at most one neighbour in $V(H)$.
\end{itemize}
\end{lemma}
\begin{proof}
%Let $u_1,u_2,u_3,u_4$ be the degree-3 vertices of $H$. Let $\{P_1,P_2\},\{Q_1,Q_2\},\{L_1,L_2\}$ be the pairs of vertex disjoint arrises of $H$, where the ends of $P_1$ are $u_1,u_2$, the ends of $P_2$ are $u_3,u_4$, the ends of $Q_1$ are $u_2,u_3$, the ends of $Q_2$ are $u_1,u_4$, the ends of $L_1$ are $u_1,u_3$, and the ends of $L_2$ are $u_2,u_4$.  See Figure \ref{H}.
Without loss of generality we may assume that $H$ is pictured as the graph in Figure \ref{H}.
First, we prove that (1) is true. Assume that $|P_1|>|P_2|$.
Since $C_1$ and $C_4$ are odd holes, $|Q_1|<|Q_2|$.
Hence, $|P_2\cup Q_1\cup L_2|<|P_1\cup Q_2\cup L_2|$, which is a contradiction to the fact that $C_2$ and $C_3$ are both odd holes. So $|P_1|=|P_2|$. By symmetry each pair of vertex-disjoint arrises have the same length. Moreover, since $C_1\Delta C_2$ is an even cycle with length at least % most
$2\ell+2$, we have $|P_1|\leq \ell$. By symmetry we have $|Q_1|, |L_1|\leq \ell$. So (1) holds.

Secondly, we prove that (2) is true. Suppose not. Since odd holes have no chord, by symmetry we may assume that there is an edge $st$ in $G$ with $s\in V(P_1^{\ast})$ and $t\in V(P_2^{\ast})$.
On  one hand, since $P_1(u_1,s)stP_2(t,u_4)Q_2$ and $P_1(u_2,s)stP_2(t,u_3)Q_1$ are cycles, by (1) we have
\[|P_1|+|P_2|+|Q_1|+|Q_2|+2=2(|P_1|+|Q_1|+1)\geq 2(2\ell+1).\] On the other hand, since $|P_1|, |Q_1|\leq \ell$ by (1), we have $|P_1|=|Q_1|=\ell$, implying that $|L_1|=|L_2|=1$. Moreover, by the symmetry between $L_1, L_2$ and $Q_1, Q_2$, we have $|Q_1|=|Q_2|=1$, which is a contradiction as $|Q_1|=\ell\geq2$. So (2) holds.

Finally, we prove that (3) is true. Suppose to the contrary that some vertex $x\in V(G)-V(H)$ has at least two neighbours in $V(H)$. Since a vertex not in an odd hole can not have two neighbours in the odd hole, $x$ has exactly two neighbours $x_1,x_2$ in $V(H)$. By symmetry we may further assume that $x_1\in V(P_1^{\ast})$ and $x_2\in V(P_2^{\ast})$.
Since $C'_1=P_1(u_1,x_1)x_1xx_2P_2(x_2,u_3)L_1$ and $C'_2=P_1(u_1,x_1)x_1xx_2P_2(x_2,u_4)Q_2$ are cycles whose lengths have different parity, $$|C'_1|+|C'_2|=2\ell+1+2(2+|P_1(u_1,x_1)|)\geq 4\ell+3.$$
Hence, $|P_1(u_1,x_1)|=\ell-1$ and $x_1u_2\in E(H)$ as $|P_1|\leq \ell$ by (1).
This implies that $u_2x_1xx_2$ is a chordal path of $C_3$ with length 3, which is a contradiction to Lemma \ref{easy case} as $\ell\geq3$.
\end{proof}

By Lemma \ref{odd k4} (1), all odd $K_4$-subdivisions of a graph $G\in \mathcal{G}_{\ell}$ have exactly $4\ell+2$ edges for  each number $\ell\geq 2$.

\section{Proof of Theorem \ref{main thm}}

Let $H_1,H_2$ be vertex disjoint induced subgraphs of a graph $G$.  An induced $(v_1,v_2)$-path $P$ is a {\em direct connection} linking $H_1$ and $H_2$ if $v_i$ is the only vertex in $V(P)$ having a neighbour in $V(H_i)$ for each $i\in \{1,2\}$. %$v_1$ is the only vertex in $V(P)$ having a neighbour in $H_1$ and $v_2$ is the only vertex in $V(P)$ having a neighbour in $H_2$.
Evidently, $V(P)\cap V(H_1\cup H_2)=\emptyset$ and the set of internal vertices of each shortest path joining $H_1$ and $H_2$ induces a direct connection linking $H_1$ and $H_2$.

For convenience, Theorem \ref{main thm} is restated here in another way.

\begin{theorem}\label{exclude odd k4}
Let $\ell\geq5$ be an integer, and $G$ be a graph in $\mathcal{G}_{\ell}$. If $G$ is $4$-vertex-critical,  then $G$ has no odd $K_4$-subdivisions.
\end{theorem}
\begin{proof}
Suppose not. Let $H$ be a subgraph of $G$ that is isomorphic to an odd $K_4$-subdivision  and pictured as the graph in  Figure \ref{H}.
By Lemma \ref{odd k4} (2), $H$ is an induced subgraph of $G$. %Let $u_1,u_2,u_3,u_4$ be the degree-3 vertices of $H$. Let $\{P_1,P_2\},\{Q_1,Q_2\},\{L_1,L_2\}$ be the pairs of vertex disjoint arrises of $H$, where the ends of $P_1$ are $u_1,u_2$, the ends of $P_2$ are $u_3,u_4$, the ends of $Q_1$ are $u_2,u_3$, the ends of $Q_2$ are $u_1,u_4$, the ends of $L_1$ are $u_1,u_3$, and the ends of $L_2$ are $u_2,u_4$.
By Lemma \ref{odd k4} (1), we have
$$|P_1|=|P_2|\leq \ell,\ \ |Q_1|=|Q_2|\leq \ell,\ \ \text{and}\ \ |L_1|=|L_2|\leq \ell. \eqno{(3.1)}$$ Without loss of generality we may assume that $P_1,P_2$ are longest arrises in $H$.
%Set $C_1:=P_1\cup Q_1\cup L_1, C_2:=P_1\cup Q_2\cup L_2$, $C_3:=P_2\cup Q_1\cup L_2$, and $C_4:=C_1\Delta C_2\Delta C_3$. Then $C_1, C_2, C_3,C_4$ are odd holes.

Let $e,f$ be the edges of $P_2$ incident with $u_3,u_4$, respectively. Since $G$ is 4-vertex-critical, $\{e,f\}$ is not an edge-cut of $G$ by Lemma \ref{2-edge-cut}, so there is a direct connection $P$ in $G\del\{e,f\}$ linking $P^*_2$ and $H-V(P^*_2)$. Let $v_1,v_2$ be the ends of $P$ with $v_2$ having a neighbour in $P^*_2$ and $v_1$ having a neighbour in $H-V(P^*_2)$. By Lemma \ref{odd k4} (3), both $v_1$ and $v_2$ have a unique neighbour in $V(H)$. Let $x,y$ be the neighbours of $v_1$ and $v_2$ in $V(H)$, respectively. That is, $x\in V(H)-V(P^*_2)$ and $y\in V(P^*_2)$.
%$v_1$ has a unique neighbour, say $x$, in $H-V(P^*_2)$ and $v_2$ has a unique neighbour, say $y$, in $P^*_2$.
Set $P':=xv_1Pv_2y$. Since $H$ is an induced subgraph of $G$, so is $H\cup P'$.

\begin{claim}\label{u1u2}
$x\notin \{u_1, u_2\}$.
\end{claim}
\begin{proof}[Subproof.]
Suppose not. By symmetry we may assume that $x=u_1$. Set $C'_4=L_1P'P_2(y,u_3)$.
Since $C_4$ is an odd hole, by symmetry we may assume that $C'_4$ is an even hole and $C_4\Delta C'_4$ is an odd hole. Since $P'P_2(y,u_3)$ is a chordal path of $C_1$, by (3.1) and Lemma \ref{easy case}, we have $|L_1|=1$. So $|P_1|=|Q_1|=\ell$ by (3.1) again. Since $P'P_2(y,u_4)$ is a chordal path of $C_2$ and $C_4\Delta C'_4$ is an odd hole, $|P'P_2(y,u_4)|=|P_1L_2|=\ell+1$ by (3.1) and Lemma \ref{easy case} again. Moreover, since $|P_2|=\ell$ and $|L_1|=1$, we have $|C'_4|\leq 2\ell$, which is not possible. So $x\neq u_1$.
\end{proof}

Set $d(H):=|P_1|-\mathrm{min}\{|Q_1|, |L_1|\}$. We say that $d(H)$ is the {\em difference} of $H$. %which is denoted by $d(H)$.
Without loss of generality we may assume that among all odd $K_4$-subdivisions, $H$ is chosen with difference as small as possible.

\begin{claim}\label{location 2}
 $x\notin V(P_1)$.
\end{claim}
\begin{proof}[Subproof.]
Suppose to the contrary that $x\in V(P_1)$. Then $x\in V(P_1^*)$ by \ref{u1u2}. Without loss of generality we may assume that $|L_1|\geq|Q_1|$. Set $C'_2=Q_2P_1(u_1,x)P'P_2(y,u_4)$.
Since $C_4$ is an odd hole, either  $C'_2$ or $C_4\Delta C'_2$ is an odd hole. Suppose that $C_4\Delta C'_2$ is an odd hole. Since $C_1\cup C_3\cup P'$ is an odd $K_4$-subdivision, by Lemma \ref{odd k4} (1) and (3.1), we have $|P'|=|Q_1|$, $|P_1(u_1,x)|=|P_2(u_4,y)|$, and $|P_1(u_2,x)|=|P_2(u_3,y)|$. So $C'_2$ is an even hole of length $2(|Q_2|+|P_1(u_1,x)|)$ by (3.1) again, implying $|L_1|+|P_1(u_1,x)|\geq |Q_2|+|P_1(u_1,x)|\geq \ell+1$ as $|L_1|\geq|Q_1|$. Then $|C_4\Delta C'_2\Delta C_1|=2|P_1(x,u_2)Q_1|\leq 2\ell$, contrary to the fact $g(G)=2\ell+1$. So $C'_2$ is an odd hole.
%$C_4\Delta C'_2\Delta C_1$ is an even cycle of length at most $2\ell$, which is a contradiction. So $C'_2$ is an odd hole.

Since $C_2\cup C'_2\cup C_3$ is an odd $K_4$-subdivision, it follows from Lemma \ref{odd k4} (1) and (3.1) that $$|P'|=|L_2|,\  |P_1(u_1,x)|=|P_2(u_3,y)|,\ \text{and}\ |P_1(u_2,x)|=|P_2(u_4,y)|. \eqno{(3.2)}$$
Then $|C_2\Delta C'_2\Delta C_1|=2|L_1|+2\ell+1$. Since $C_2\Delta C'_2\Delta C_1$ is not an odd hole,
%Since $C_2\Delta C'_2\Delta C_1$ is an odd cycle of length larger than $2\ell+1$, it is not an odd hole,
$$1\in\{|Q_2|, |P_1(u_2,x)|,|P_2(u_3,y)|\}. \eqno{(3.3)}$$ When $|P_1(u_2,x)|=1$, since $|C_2\Delta C'_2\Delta C_1|=2|L_1|+2\ell+1$ and $g(G)=2\ell+1$, we have $|L_1|=|P'|=\ell$ by (3.2), implying $|P_1|=\ell$ and $|Q_1|=1$ as $P_1,P_2$ are longest arrises in $H$. Hence, $d(H)=\ell-1$. %the difference of $H$ is $\ell-1$.
%When $|P_1(u_2,x)|=1$, by (3.1), (3.2) and Lemma \ref{easy case}, we have $|L_1|=|P_1|=|P'|=\ell$, so
%Since $|P_1(u_2,x)|=|P_2(u_4,y)|=|Q_1|=1$, the graph
Then $G[V(C_1\cup C'_2 \cup P_2)]$ is an odd $K_4$-subdivision with difference $\ell-2$, which is a contradiction to the choice of $H$. So $|P_1(u_2,x)|\geq2$. Assume that $|Q_2|=1$. Then $|L_1|=|P_1|=\ell$ by (3.1). Since $|P_1(u_2,x)|\geq2$, the graph $G[V(C'_2\cup C_2\cup C_3)]$ is an odd $K_4$-subdivision whose difference is at most $\ell-2$, which is a contradiction to the choice of $H$ as $d(H)=\ell-1$. So $|Q_2|\geq2$. Then $yu_3\in E(H)$ by (3.3), implying $xu_1\in E(H)$ by (3.2).  Hence, $|C_4\Delta C'_2|=2+2|L_1|$ by (3.1) and (3.2), and so $|L_1|=\ell$ by (3.1) again. Since $|P_1|\geq |L_1|$, we have $|P_1|=\ell$ and $|Q_1|=1$ by (3.1), which is a contradiction as $|Q_2|\geq2$.
%Since $G|V(C'_2\cup C_2\cup C_3)$ is an odd $K_4$-subdivision whose arrises have lengths in $\{|P_1|-1, |L_1|, |Q_1|+1\}$, by the choice of $H$, we have $$|L_1|=|P_1|\ \text{and}\ |P_1|\in\{|Q_1|, |Q_1|+1\}. \eqno{(3.4)}$$ Since $\ell\geq3$, we have $|Q_1|\geq2$ and $|P'|=|L_1|\geq3$.
%Assume that $\{x,y\}$ is a vertex cut of $G$. Let $G_1$ be the induced subgraph of $G$ whose vertex set consists of $\{x,y\}$ and the vertex set of the component of $G-\{x,y\}$ containing $P$. Since $|P'|\geq3$, the graph $G_1$ is well defined. Since all induced $(x,y)$-paths in $G_1$ has the same length as $P'$, %either some vertex in $P$ has degree 2 or the ends of an edge in $P$ is a vertex cut of $G_1$ by Lemma \ref{same length}. So
\end{proof}

\begin{claim}\label{location of x}
When $x\in \{u_3,u_4\}$, the vertices $x$ and $y$ are adjacent, that is, $xy\in\{e,f\}$.
\end{claim}
\begin{proof}[Subproof.]
By symmetry we may assume that $x=u_3$. Assume to the contrary that $x, y$ are not adjacent. Set $C'_3=P'P_2(y,u_3)$. Since $P'$ is a chordal path of $C_3$, we have that $C_3'$ is an odd hole by Lemma \ref{easy case} and (3.1). Since $C_3'\Delta C_3$ is an even hole, $|Q_1|=|L_2|=1$ by (3.1) and Lemma \ref{easy case} again. Then $|P_1|=2\ell-1>\ell$, which is a contradiction to (3.1). So $e=xy$.
\end{proof}

\begin{claim}\label{location of x-}
When $x\in V(L_1^*)$, we have that $|Q_1|=1$, $|P_1|=|L_1|=\ell$, $|P'|=2\ell-1$, $xu_3,yu_3\in E(H)$ and $xu_3yP'$ is an odd hole.
\end{claim}
\begin{proof}[Subproof.]
%By \ref{u1u2}, we have $x\in V(L_1^*)$.
Set $C'_4=L_1(x,u_1)Q_2P_2(u_4,y)P'$. We claim that $C_4\Delta C_4'$ is an odd hole.
Assume to the contrary that $C_4\Delta C_4'$ is an even hole. Since $x\neq u_3$, the subgraph $C_1\cup(C_4\Delta C_4')$ is an induced theta subgraph. Hence, $xu_3\in E(H)$ by (3.1) and Lemma \ref{easy case}. Similarly, $yu_3\in E(H)$. Since $P'$ is a chordal path of $C_4$, we get a contradiction to Lemma \ref{easy case}. So the claim holds, implying that $C'_4$ is an even hole.
%Since $x\neq u_1$, the path $u_3xP'$ is a chordal path of $C_3$, so $yu_3\in E(H)$ by (3.1) and Lemma \ref{easy case} again. Hence, $C_4\Delta C'_4$ is an odd hole by Lemma \ref{easy case}, which is a contradiction.%If $yu_3\notin E(H)$, then $|u_3xP'|=|P_2(u_3,y)|\geq\ell+1$ by (3.1) and Lemma \ref{easy case}, which a contradiction to the fact that $|P_2|\leq\ell$. If $yu_3\in E(H)$, then $C'_4$ is an odd hole of length at least  $2\ell+3$, which is not possible. So $C_4\Delta C'_4$ is an odd hole, implying that $C'_4$ is an even hole.

Since $x\neq u_3$, the graph $C_2\cup C'_4$ is an induced theta subgraph of $G$. Moreover, since $C'_4$ is an even hole, $|Q_2|=1$ by (3.1) and Lemma \ref{easy case}. Hence, $|P_1|=|L_1|=\ell$ by (3.1) again. Assume that $y, u_3$ are not adjacent. Since $C_1\cup C'_4$ is an induced theta subgraph of $G$, we have $xu_1\in E(H)$, implying $|P(x,u_3)|=\ell-1$. Since $P'$ is a chordal path of $C_4$ and $C_4\Delta C_4'$ is an odd hole, $yu_3\in E(H)$ by Lemma \ref{easy case}, a contradiction. Hence, $yu_3\in E(H)$. By symmetry we have $xu_3\in E(H)$. This proves \ref{location of x-}.
%Assume that $x, u_3$ are not adjacent. Since $|P_1|\geq2$, the path $P'L_1(x,u_1)Q_2$ is a chordal path of $C_3$, so $C'_4$ is an odd hole by Lemma \ref{easy case} and (3.1), a contradiction. So $xu_3\in E(H)$. This proves \ref{location of x-}.
%implying that $C_3\cup C'_4$ is an induced theta subgraph of $G$. Then $yu_4\in E(H)$. % Hence, $xu_1u_4y$ is a chordal path of $C_4\Delta C'_4$ with length 3, which is not possible by Lemma \ref{easy case}.
%Since $C_4$ and $C_4\Delta C_4'$ are odd holes, $|P'|=3$, so $|C'_4|=6$, which is not possible. So $yu_3\in E(H)$.
\end{proof}

%So $|L_1|\geq2$. Let $u'_1$ be the neighbour of $u_1$ in $L_1$.
\begin{claim}\label{location 3}
Assume that $P'$ has the structure stated as in \ref{location of x-}. Then no vertex in $V(G)-V(H\cup P')$ has two neighbours in $H\cup P'$.
\end{claim}
\begin{proof}[Subproof.]
Assume to the contrary that some vertex $a\in V(G)-V(H\cup P')$ has two neighbours $a_1,a_2$ in $H\cup P'$. Since no vertex has two neighbours in an odd hole, it follows from Lemma \ref{odd k4} (3) and \ref{location of x-} that $a$ has exactly two neighbours in $H\cup P'$ with $a_1\in V(H)-\{x,y,u_3\}$ and $a_2\in V(P)$. When $xa_2\notin E(P')$, let $Q$ be the unique $(y,a_1)$-path in $G[V(P)\cup\{a_1,y\}]$. Since $Q^*$ is a direct connection in $G\del\{e,f\}$ linking $P^*_2$ and $H-V(P^*_2)$, by \ref{location 2}-\ref{location of x-} and the symmetry between $P'$ and $Q$, we have $a_1\in\{x,u_3\}$, contrary to the fact $a_1\in V(H)-\{x,y,u_3\}$. So $xa_2\in E(P')$.
%there is another direct connection $Q$ in $G\del\{e,f\}$ linking $P^*_2$ and $H-V(P^*_2)$ with $V(Q)\subset V(P)\cup\{a\}$ such that $G[V(Q)\cup\{a_1,y\}]$ is an induced path and %with ends $y$ and $a_1$ with $Q$ as its interior
%By \ref{u1u2}-\ref{location of x-} and symmetry, we have $xa_2\in E(P')$.
Moreover, since $|P_1|=|L_1|=\ell\geq5$ and $g(G)=2\ell+1$, we have $a_1\notin V(P_1)$. Let $u'_1$ be the neighbour of $u_1$ in $L_1$. When $a_1\in V(C_2)-V(P_1)$, since $aa_2$ is a direct connection in $G\del\{u_1u'_1,u_3x\}$ linking $L_1^*$ and $H-V(L_1^*)$, which is not possible by \ref{location of x-} and the symmetry between $P_2$ and $L_1$. So $a\in V(L_1^*)$. Then $xa_2aa_1$ is a chordal path of $C_1$ with length 3, contrary to Lemma \ref{easy case}.
%by \ref{u1u2}-\ref{location of x-} and the symmetry between $P_2$ and $L_1$, without loss of generality we may assume that
\end{proof}

\begin{claim}\label{gap 1}
$x\in \{u_3,u_4\}$ and $xy\in\{e,f\}$.
\end{claim}
\begin{proof}[Subproof.]
%Without loss of generality we may assume that $|L_1|\geq|Q_1|$.
By \ref{u1u2}-\ref{location of x}, it suffices to show that $x\notin V(L_1^*\cup L_2^*\cup Q_1^*\cup Q_2^*)$. Assume not. By symmetry  we may assume that $x\in V(L_1^*)$. By \ref{location of x-}, we have that $$xu_3\in E(L_1),\ e=yu_3,\ |P'|=2\ell-1,\ |P_1|=|L_1|=\ell,\ \text{and}\ |Q_1|=1.$$
Since no 4-vertex-critical graph has a $P_3$-cut by Lemma \ref{P3}, to prove that \ref{gap 1} is true, it suffice to show that $\{x,y,u_3\}$ is a $P_3$-cut of $G$.  Suppose not. Let $R$ be a shortest induced path in $G-\{x,y,u_3\}$ linking $P$ and $H-\{x,y,u_3\}$.  Let $s$ and $t$ be the ends of $R$ with $s\in V(P)$. By \ref{location 3}, $|R|\geq3$ and no vertex in $V(H\cup P')-\{x,y,u_3,s,t\}$ has a neighbour in $R^*$.

We claim that $t\notin V(L_1\cup P_2)-\{x,y,u_3\}$. Suppose not. By symmetry  we may assume that $t\in V(L_1)-\{x, u_3\}$. Let $R_1$ be the induced $(y,t)$-path in $G[V(P'\cup R)-\{x\}]$. When $u_3$ has no neighbour in $R_1^*$, set $R_2:=R_1$ and $C:=R_2L_1(t,u_3)u_3y$. When $u_3$ has a neighbour in $R_1^*$, let $t'\in V(R_1^*)$ be a neighbour of $u_3$ closest to $t$ and set $R_2:=u_3t'R_1(t',t)$ and $C:=R_2L_1(t,u_3)$. Note that $C_4\Delta C$ is a hole, although $C$ maybe not a hole.
Since $C\Delta C_1 \Delta C_2$ is an odd hole with length at least $2\ell+3$ when $C$ is an odd cycle, to prove the claim, it suffices to show that $|C|$ is odd. When $x$ has a neighbour in $R_2^*$, since $|R_2|\geq2\ell$ by (3.1) and the fact that $g(G)=2\ell+1$, the subgraph $C_4\Delta C$ is an even hole, which implies that $C$ is an odd cycle. So we may assume that $x$ has no neighbour in $R_2^*$. When $u_3$ is an end of $R_2$, since $R_2$ is a chordal path of $C_1$, it follows from Lemma \ref{easy case} and (3.1) that $C$ is an odd hole. When $y$ is an end of $R_2$ and $u_3yR_2$ is a chordal path of $C_1$, for the similar reason, $C$ is an odd hole. Hence, we may assume that $y$ is an end of $R_2$ and $u_3yR_2$ is not a chordal path of $C_1$, implying $xs\in E(P')$ and $s\in V(R_2)$. Since $P\subset R_2$, we have $|R_2|>2\ell$, so $C_4\Delta C$ is an even hole, implying that $C$ is an odd cycle. Hence, this proves the claim.
%since neither $x$ nor $u_3$ has a neighbour in $R_2^*$, we have $s=v_1$ by \ref{u1u2} and \ref{location of x-}, so $|R_2|>2\ell$. Then $C_4\Delta C$ is an even hole, so $C$ is an odd cycle. Hence, the claim holds.

By symmetry we may therefore assume that $t\in V(P_1)-\{u_1\}$. Let $R_1$ be the induced $(y,t)$-path in $G[V(P'\cup R)-\{x\}]$. By \ref{u1u2} and \ref{location 2}, either $xs\in E(P')$ and $y$ has no neighbour in $R$ or some vertex in $\{x, u_3\}$ has a neighbour in $R_1^*$. No matter which case happens, we have $|R_1|\geq 2\ell$. That is, $R_1P_2(y,u_4)$ is a chordal path of $C_2$ with length at least $3\ell-1$, which is a contradiction to Lemma \ref{easy case} as $t, u_4$ are non-adjacent. Hence, $\{x,y,u_3\}$ is a $P_3$-cut of $G$.
\end{proof}

By \ref{gap 1}, there is a minimal vertex cut $X$ of $G$ with $\{u_3,u_4\}\subseteq X\subset N[\{u_3,u_4\}]$ and $\{u_3,u_4\}=X\cap V(H)$. Let $G_1$ be the induced subgraph of $G$ whose vertex set consists of $X$ and the vertex set of the component of $G-X$ containing $P_2^{\ast}$. Since $\ell\geq5$, we have $|P_2|\geq4$ by (3.1). If all induced $(u_3,u_4)$-paths in $G_1$ have length $|P_2|$, by Lemma \ref{2-vertex-cut}, $G$ has a degree-2 vertex, a $K_1$-cut or a $K_2$-cut, which is not possible as $G$ is 4-vertex-critical. Hence, to finish the proof of Theorem \ref{exclude odd k4}, it suffices to show that all induced $(u_3,u_4)$-paths in $G_1$ have length $|P_2|$.

Let $Q$ be an arbitrary induced $(u_3,u_4)$-path in $G_1$. When $|L_1|\geq2$, since $QQ_2$ is a chordal path of $C_1$ by Lemma \ref{odd k4} (3) and the definition of $G_1$, we have $|QQ_2|=|Q_1P_1|$ by Lemma \ref{easy case}, so $|Q|=|P_1|$ by (3.1). Hence, by (3.1) we may assume that $|L_1|=1$ and $|Q_1|=|P_1|=\ell$.
Since $Q_1L_2$ is an induced $(u_3,u_4)$-path of length $\ell+1$, either $|Q|=|P_2|=\ell$ or $|Q|\geq\ell+1$ and $|Q|$ has the same parity as $\ell+1$. Assume that the latter case happens. Without loss of generality we may further assume that $Q$ is chosen with length at least $\ell+1$ and with $|P_2\cup Q|$ as small as possible. Since $|Q|$ and $|P_2|$ have different parity, $P_2\cup Q$ is not bipartite. Moreover, by the choice of $Q$, the subgraph $P_2\cup Q$ contains a unique cycle $C$ and $|C|$ is odd. Since $Q=P_2\Delta C$ is an induced path of length at least $\ell+1$, we have $|C\cap Q|>|C\cap P_2|\geq2$. So $C\Delta C_3\Delta C_1$ is an odd hole of length at least $2\ell+3$, which is not possible.
\end{proof}

\section{Acknowledgments}
%This research was partially supported by grants from the National Natural Sciences Foundation of China (No. 11971111).
The authors thank the two referees for their careful reading of this manuscript and  pointing out an error in our original version. % Theorem \ref{exclude odd k4} not directly implying the original Theorem 1.2. Luckily, we only need Theorem \ref{exclude odd k4} to prove the main theorem in \cite{Chen22}.
%\section{Acknowledgments}

%The authors thank both referees for carefully reading the paper.

\end{document}